\documentclass[11pt]{amsart}
\usepackage[english]{babel}
\usepackage{amsmath,amsfonts,amsthm,amssymb,amsbsy,upref,color,graphicx,hyperref,enumerate,comment,tikz}
\usepackage{stackengine, mathrsfs}
\usepackage{esint}
\usepackage{mathtools}
\usepackage[a4paper,margin=2.70truecm]{geometry}
\usepackage{bbm}
\usepackage{soul}
\usepackage[hyperpageref]{backref}

\DeclareMathOperator{\dist}{dist}
\DeclareMathOperator{\supp}{supp}

\def\eps{{\varepsilon}}

\def\O{\Omega}

\def\R{\mathbb{R}}

\def\F{\mathcal{F}}

\def\H{{\mathcal H}}

\def\eps{\varepsilon}

\newcommand{\be}{\begin{equation}}
\newcommand{\ee}{\end{equation}}
\newcommand{\bib}[4]{\bibitem{#1}{\sc#2: }{\it#3. }{#4.}}

\numberwithin{equation}{section}
\theoremstyle{plain}

\newtheorem{theo}{Theorem}[section]
\newtheorem{lemm}[theo]{Lemma}

\newtheorem{prop}[theo]{Proposition}

\theoremstyle{definition}
\newtheorem{rema}[theo]{Remark}

\title[BMO-type functionals]{BMO-type functionals related to the total variation and connection to denoising models}

\author[S. Guarino Lo Bianco]{Serena Guarino Lo Bianco}
\author[R. Schiattarella]{Roberta Schiattarella}

\date{}

\begin{document}

\maketitle

\begin{center}
    {\it Dedicated to G.Buttazzo on the occasion of his 70'th birthday\footnote{This paper is dedicated to Professor G. Buttazzo who inspired this reasearch.These results are an outgrowth of an informal discussion that took place during the workshop \lq\lq Weekend di lavoro su Calcolo delle Variazioni" held in Montecatini Terme (PT) in November 2022.}}
\end{center}

\begin{abstract}
The purpose of this paper is to analyze the asymptotic behaviour in the spirit of $\Gamma$-convergence of BMO-type functionals related to the total variation of a function $u$. Moreover, we deal with a minimization problem coming from applications in image processing. 
\end{abstract}



\section{Introduction}

Recently the study of characterizations of Sobolev and bounded variation functions by certain BMO-type seminorms has emerged as an intriguing research area.
In \cite{BBM} the authors, inspired by the celebrated space of John and Nirenberg of bounded mean oscillation ($BMO$), introduced the space $B$ of functions from the unit cube $Q= \left(-\frac{1}{2}, \frac{1}{2} \right)^n$ such that the following seminorm is ﬁnite
\[
[u]_B:= \sup_{0< \eps <1} [u]_\eps, \quad \text{where\,\,} [u]_\eps=\eps^{n-1} \sup_{\mathcal G_\eps} \sum_{Q'\in \mathcal G_\eps}  \fint_{Q'}\left|u(x)-\fint_{Q'}u\right|\,dx.
\]
Here $\mathcal G_\eps$ denotes a collection of mutually disjoint $\eps$-cubes $Q'$ of the type $Q^\prime=x+\eps Q$ whose cardinality does not exceed $\eps^{1-n}$. The function space $B$ contains in particular $BMO$, the space of bounded variation functions $BV$ and the fractional space $W^{\frac1 p,p}$ for $1 \leq p < \infty$.

Many variants have been considered in the last years. In particular when the family involves $\eps$-cubes of general orientation (this is the \textit{isotropic case}), it is proved in \cite{FMS1} and \cite{DFP} that given an open set $\Omega \subset \R^n$ and a function $u \in SBV (\Omega)$, the space of special $BV$ functions, then
\be\label{theoFMS}
\lim_{\eps\rightarrow 0} \eps^{n-1} \sup_{\mathcal J_\eps} \sum_{Q'\in \mathcal J_\eps}  \fint_{Q'}\left|u(x)-\fint_{Q'}u\right|\,dx = \frac 1 4 \int_\O |\nabla u|\, dx + \frac 1 2 |D^s u|(\O)\,.
\ee
Here $\mathcal{J}_\eps$ is a family of pairwise disjoint cubes of side length $\eps$ contained in $\Omega$. By considering in \eqref{theoFMS} the characteristic function of a measurable set $A\subset \R^n$ a characterization of ﬁnite perimeter $P(A)$ is obtained (this was originally proved in \cite{ABBF}).

The \textit{anisotropic} version of \eqref{theoFMS} was considered in \cite{AC} and \cite{FFGS} (see also \cite{FGS}). Given $D$ a bounded and connected open set with Lipschitz boundary and a function $u\in L^1(\Omega)$, for any $\eps>0$ we consider the following functional
\begin{equation}\label{Hfunct}
H_\eps(u, \Omega) = \eps^{n-1} \sup_{\H_\eps} \sum_{D'\in\mathcal H_\eps}  \fint_{D'} \left|u(x)-\fint_{D'}u \right|\, dx
\end{equation}
where $\H_\eps$ is a family of pairwise disjoint translations $D'$ of $\eps D$ contained in $\Omega$.\\  
In \cite{FFGS} it is proved that if $D$ is a bounded open set satisfying some mild regularity assumptions and $u\in SBV(\Omega)$, then there exist two Lipschitz continuous 1-homogeneous  functions $\varphi,\,\psi:\mathbb{R}^{n}\to (0,+\infty)$, $\psi\leq\varphi$, strictly positive on $\R^{n}\setminus\{0\}$, $\psi$ convex, such that
\be\label{theoFFGS}
\lim_{\eps\to 0} H_\eps(u,\O) = \int_{\Omega} \psi(\nabla u(x))\, dx + \int_{J_u}\left(u^+(x)-u^-(x)\right)\varphi(\nu_u(x))\,d\H^{n-1}(x)\,.
\ee
In \eqref{theoFFGS} $\nabla u$ stands for the absolutely continuous part of the gradient measure $Du$, $J_u$ is the jump set of $u$, $u^+>u^-$ are the traces of $u$ on both sides of $J_u$ and $\nu_u$ is the generalized normal to $J_u$ oriented in the direction going from $u^-$ to $u^+$. The particular case $u=\chi_A$ where $A\subset \R^n$ is a set of finite perimeter, was studied in \cite{AC}.\\
One may then be tempted to infer that the same conclusion holds for every $u\in BV(\Omega)$. In fact, the quantity $H_\eps(u, \Omega)$ is strictly related to the total variation $|Du|(\Omega)$ of $u$ (see \cite{FFGS}). Indeed, by using H\"older inequality and Poincar\'e-Wirtinger inequality, for any $u\in BV$, we get that there exists a constant $C>0$ depending only on $D$ such that if $D'=x_0+ \eps D$, then 
$$
\eps^{n-1} \sum_{D'\in\mathcal H_\eps}  \fint_{D'} \left|u(x)-\fint_{D'}u \right|\, dx\leq C |Du|(\Omega)
$$
and thus, 
\begin{equation}\label{Poinc}
H_\eps(u, \Omega) \leq C |Du|(\Omega).
\end{equation}
However, for general $BV$ functions $u$, the point-wise limit of $H_\eps(u, \Omega)$ is more difficult to grasp. In particular, for functions $u\in BV(\Omega)\setminus SBV(\Omega)$ having the so-called Cantor part of the derivative, it is possible that the limit of $H_\eps(u, \Omega)$ as $\eps$ goes to $0$ does not exist, as shown by a one dimensional example of \cite{FMS1}. 
Nevertheless, one can still characterize the functions in $BV(\Omega)$ as the functions $u\in L^1(\Omega)$ such that 
$\limsup_{\eps\to 0^+} H_\eps(u, \Omega)<+ \infty$ (see \eqref{charBV} below).\\
The non existence of the point-wise limit for general $BV$ functions suggests that the mode of convergence of $H_\eps$ to the total variation as $\eps$ goes to $0$ is extremely delicate. It is natural to expect that the appropriate framework in this case to analyze the asymptotic behaviour of $H_\eps$ is the $\Gamma$- limit (in the sense of E. De Giorgi). Obviously, since we are considering an anisotropic variant of the BMO-type seminorm by using, instead of cubes, covering families made by translations of a given set $D$, as $\Gamma$-limit we expect an anisotropic version of the total variation. For this reason, by considering the 1- homogeneous, Lipschitz function $\psi: \Omega \to (0,+ \infty)$ that appears in \eqref{theoFFGS}, we define 
\begin{equation*}
\Psi(Du)(\Omega)= \inf \left\{  \liminf_{h\to\infty} \int_\Omega \psi(\nabla u_h)\, :\, u_h\in C^\infty(\Omega), u_h\to u \text{ in } L^1(\Omega ) \right\}.
\end{equation*}
This definition with $\psi=|\cdot|$ coincides with the usual $BV$ total variation (see Theorem 3.9. in \cite{AFP}). We observe that if $u\in BV(\Omega)$ then $\Psi(Du)(\Omega)<+ \infty$ (see \eqref{equivnorm}).

The main result of this paper reads as follows.
\begin{theo}\label{gammaH}
	The family of functionals $(H_\eps)$ defined in \eqref{Hfunct} for $\eps>0$, $\Gamma$-converges in $L^1$ to the functional $H$ defined for any $u\in L^1(\Omega)$ by
$$
H(u,\Omega)= \begin{cases}
 \Psi(Du) & \text{ if } u\in BV(\Omega)\\
 + \infty  & \text{ otherwise}.
\end{cases}
$$
\end{theo}

Finally, the {\it higher order and isotropic} counterpart of $H_\eps$ is the following functional defined for functions $u$ in the higher Sobolev space $W^{m-1,1}_{\text{loc}}(\Omega)$, $m\in \mathbb{N}$, for any $\eps>0$, 
\be\label{Keps}
K_\eps(u,m, \O)= \eps^{n-m} \sup_{\mathcal{G}_\eps}\sum_{Q^\prime\in \mathcal{G}_\eps} \fint_{Q^\prime} \left| u(x) -P^{m-1}_{Q^\prime} [u](x) \right|\, dx, 
\ee
where the families $\mathcal G_\eps$ are made of disjoint cubes $Q^\prime= x_0+\eps Q$ of side length $\eps$, centered in $x_0$, with arbitrary orientation contained in $\O$ and $P^{m-1}_{Q^\prime}[u]$ is the polynomial of degree $m-1$ centered at $x_0$, given by
\begin{equation}\label{poly}
P^{m-1}_{Q^\prime}[u](x)= \sum_{|\alpha|\leq m-1} (x-x_0)^\alpha \fint_{Q'}{\left(D^{\alpha}u\right)(s)}\, ds.
\end{equation}
In \cite{GS} it was proved that  if $u\in W^{m,1}_{\text{loc}}(\Omega)$ then
\be\label{theoGS}
		\lim_{\eps\to 0} K_\eps(u,m, \O) = \beta(n,m)\int_{\O} |\nabla^m u|\, dx,
\ee
where
\begin{equation}\label{beta}
	\beta(n,m):=  \max_{\nu \in \mathbb S^{N-1}} \frac{1}{m!}\int_{Q} \Big\lvert\nu\cdot x^m- \int_Q \nu\cdot y^m\, dy \Big \lvert\, dx.
\end{equation}
We refer to Section 2 of \cite{GS} for the notation used in \eqref{beta}. Observe that this last result is the extension of Theorem 2.2 in \cite{FMS2}.

Also, the point-wise limit of $K_\eps$ for functions  in $BV^m$ does not exists but it is possible to characterize the functions in $BV^m(\Omega)$, the space of functions of m-th order bounded variation (see \cite{DM}), as the functions such that $\limsup_{\eps\to 0} K_\eps(u,m,\Omega)<+\infty$ (see \eqref{charBVm} below). Still using the $\Gamma$-convergence framework, we have the following result.
\begin{theo}\label{gammaK}
	The family of functionals $(K_\eps)$ defined in \eqref{Keps} for $\eps>0$, $\Gamma$-converges in $L^1$ to the functional $K$ defined for any $u\in L^1(\Omega)$ by
$$
\Gamma-\lim K_\eps(u,m,\Omega)= \begin{cases}
 \beta(m,n)|D^m u|(\Omega)& \text{ if } u\in BV^m(\Omega)\\
 + \infty  & \text{ otherwise}
\end{cases}
$$
\end{theo}

When $m=1$ this result was established in \cite{ARB}. The used techniques are different; their approach is based on piece-wise constant
approximations rather than convolutions and our method makes the demonstration significantly shorter.

In the last section of the paper, motivated by applications in Image Processing, we deal with the minimization of functionals of the form
\[F_\eps(u,\Omega)= \Lambda \int_\Omega |f-u|^q + K_\eps(u,1,\Omega), \qquad q\geq 1
\]
i.e. functionals that are the sum of a fidelity part and of a regularization term. Here for the sake of brevity, we denote $K_\eps(u,\Omega)=K_\eps(u,1,\Omega)$. We investigate the existence of a minimizer when $\eps$ is fixed and we analyze their behavior as $\eps \to 0$. The proof of the existence of minimizer is based on the $\Gamma -$convergence result (Theorem \ref{gammaK}). Precisely, we prove

\begin{theo}\label{mainappl}
	Let $q>1$. For every $\eps >0$, there exists a unique $u_\eps\in L^q(\Omega)$ such that
	\[
	F_\eps(u_\eps,\Omega)=\min_{u\in L^q(\Omega)} F_\eps(u,\Omega)
	\]
	Let $u_0$ the unique minimizer of $F$ in $L^q(\Omega)\cap BV(\Omega)$, then as $\eps\to 0$ we have that
	\[
	u_\eps \to u_0, \qquad \text{in } L^q(\Omega)
	\]
	and
	\[
	F_\eps(u_\eps,\Omega) \to F(u_0,\Omega),
	\]
 where 
\[
F(u_0,\Omega)=\frac 1 4 |Du_0|(\Omega)  + \Lambda\int_\Omega |f-u_0|^q.
\]
\end{theo}
In the case $q=1$ one can not apply the previous Theorem but one can always consider almost minimizers; a slight generalization is given in Theorem \ref{quasimin}.

\section{Preliminaries}

We collect some preliminary results and properties of the functionals $H_\eps$ and $K_\eps$ useful in the next sections.\par
Here and in the rest of the paper $\Omega$ will be an open set in $\R^n$. We denote by $Du$ the gradient measure of $u$ and by $\nabla u$ its absolutely continuous part.

We recall the definition of $BV$ functions. A function $u\in L^1(\Omega)$ is said to have bounded variation, $u\in BV(\Omega)$, if
$$
|Du|(\Omega):= \sup\left\{\int_\O u \text{ div} \phi \,dx: \phi \in C^1_0(\Omega, \mathbb R^n), \| \phi\|_{\infty}\leq 1\right\}<+ \infty.
$$
$BV$ is a Banach space endowed with the norm
$$
\| u\|_{BV}: = \|u \|_{L^1(\O)} + |Du|(\Omega).
$$
Smooth functions are not dense in $BV(\Omega)$, but every function $u\in BV(\Omega)$ is approximable in a weak sense (see Theorem 3.9 in \cite{AFP}); that is, there exists a sequence $(u_h)_h$ of $C^{\infty}(\Omega)$ functions such that
\begin{equation}\label{star}
\begin{cases}
    \|u_h- u\|_{L^1(\Omega)} \to 0\\
    \\
    \int_{\Omega} |\nabla u_h| \to |Du|(\O).
\end{cases}
\end{equation}
\subsection{Properties of the functional $H_\eps$}
Given two functions $u, v\in L^1(\Omega)$, the following properties hold true:
\begin{itemize}
	\item an estimate from above\be\label{upperdiff}|H_\eps(u, \Omega)- H_\eps(v, \Omega)|\leq H_\eps(u-v, \Omega);\ee
	\item convexity with respect to the function: for any $\lambda\in [0,1]$, we have $$H_\eps(\lambda u+ (1-\lambda) v, \Omega) \leq \lambda H_\eps(u, \Omega)+(1-\lambda) H_\eps (v, \Omega). $$
\end{itemize}

\begin{lemm}\label{lemmaconvol}
   For every open set $A\subset \subset \Omega$, let $0<\sigma <\text{dist}(A, \partial \Omega)$, then
\begin{equation}\label{convol}
H_\eps (\rho_{\sigma}* u, A)\leq H_\eps ( u, \Omega) \qquad \forall u\in L^1(\Omega),\end{equation}
where $\rho_\sigma(x)= \sigma^{-n} \rho \left(\frac{x}{\sigma} \right)$ and $\rho$ is a standard mollifier with compact support in the unit ball $B$.
\end{lemm}
 \begin{proof}
By fixing an open set $A\subset \subset \Omega$ and $0<\sigma <\text{dist}(A, \partial \Omega)$, for all $x\in A$, we set $u_\sigma(x):= (\rho_{\sigma}*u)(x)$. Thus, given a family $\mathcal H_\eps$ of pairwise disjoint sets $D'$ translations of $\eps D$ contained in $A$, using the definition of $u_\sigma$, Jensen inequality and Fubini's theorem we have that, recalling that $\int_B \rho(y)\, dy=1$,
\begin{equation*}
\begin{split}
\eps^{n-1} \sum_{D'\in\mathcal H_\eps} & \fint_{D'} \left|u_{\sigma}(x)-\fint_{D'}u_{\sigma} \right|\, dx\\
&=
\eps^{n-1} \sum_{D'\in\mathcal H_\eps}  \fint_{D'} \left| \int_B \rho(y) u(x-\sigma y)\,dy -\fint_{D'}\int_{B} \rho(y) u(z- \sigma y)\, dy dz\right|\, dx\\
&\leq
\eps^{n-1} \int_{B} \rho(y) \left(\sum_{D'\in\mathcal H_\eps}  \fint_{D'} \Big|u(x-\sigma y)\,dy -\fint_{D'} u(z- \sigma y) dz\Big|\, dx\right) \, dy\\
&=
\eps^{n-1} \int_{B} \rho(y) \left(\sum_{D'\in\mathcal H_\eps}  \fint_{D'-\sigma y} \Big|u(x) -\fint_{D'-\sigma y} u\Big|\, dx\right)\, dy\leq H_{\eps} (u, \Omega).
\end{split}
\end{equation*}
Taking the supremum over all families $\mathcal H_\eps$, we get \eqref{convol}.

\end{proof}

In the particular case that $u$ is the linear function, we set for $\nu\in \mathbb S^{n-1}$, 
$$
\psi (\nu):= H(x\cdot \nu, Q)
$$
where $Q= \left(-\frac{1}{2},\frac{1}{2} \right)^{n}$ is the unit cube and $H(x\cdot \nu, Q)=\lim_{\eps \to 0} H_\eps(x\cdot \nu, Q) $. The function $\psi$ is well defined  as showed in Section 3.2. of \cite{FFGS}. Moreover, the function $\psi$ is Lipschitz continuous, bounded away from zero and convex (see Propositions 3.4 and 3.6 of \cite{FFGS}). With a slightly abuse of notion, we shall denote by $\psi$ the $1$- homogeneous extension $\tilde{\psi}$ of $\psi$ to $\mathbb R^n$, defined in the following way:
$$
\tilde{\psi}(0)=0
$$
and
$$\tilde{\psi}(\tau)= |\tau|\psi\left( \frac{\tau}{|\tau|} \right) \qquad \forall \tau \in \mathbb R^{n}\setminus \{0\}.$$

Clearly, the $1$-homogeneous extension of $\psi$ is Lipschitz and there exists a positive constant $c$ such that $$\psi(\tau)\geq c|\tau|\qquad \forall \tau \in \mathbb R^{n}\setminus \{0\}.$$  Moreover, if $u\in C^\infty(\Omega)$ there exist two positive constants $C_1$, $C_2$ such that
\be\label{equivpsi}
C_1\int_{\Omega} |\nabla u(x)| \, dx \leq  \int_{\Omega} \psi(\nabla u(x)) \, dx\leq C_2 \int_{\Omega} |\nabla u(x)|\, dx.
\ee

We summarize here some results of \cite{FFGS}:

\begin{theo}\label{w11lim1}
Let $D$ be a bounded connected open set with Lipschitz boundary. Then
\begin{itemize}
\item[a)] 
if $u\in W^{1,1}(\Omega)$, then 
\begin{equation}
\lim_{\eps \to 0} H_\eps(u,\Omega)=\int_{\Omega} \psi (\nabla u)\, dx\,
\end{equation}
\item[b)] 
if $u\in L^1(\Omega)$, then $u \in BV(\Omega)$ if, and only if, 
\begin{equation}\label{charBV}
 \liminf_{\eps \to 0} H_{\eps} (u, \Omega) <+\infty
\end{equation}
\end{itemize}

 \end{theo}

We introduce on $BV$ an anisotropic norm equivalent to the total variation. 
We define the anisotropic variation $\Psi(Du)$ defined on open sets $A\subseteq \Omega$ by
\begin{equation}\label{defPsi}
\Psi(Du)(A)= \inf \left\{  \liminf_{h\to\infty} \int_A \psi(\nabla u_h)\, :\, u_h\in C^\infty(A), u_h\to u \text{ in } L^1(A ) \right\}.
\end{equation}
Moreover, Reshetnyak continuity theorem yields
\begin{equation}\label{Resh}
\Psi(Du)(A)=\int_A \psi\left(\frac{d Du}{d|Du|} \right)\, d|Du|.
\end{equation}

Clearly, this definition with $\psi=|\cdot|$ coincides with the usual $BV$ total variation (see Theorem 3.9. in \cite{AFP}). 
\par
For any $u\in BV$, from \eqref{Resh} and by \eqref{star}, \eqref{equivpsi} we get that there exists two positive constants $C_1$, $C_2$ such that
\begin{equation} \label{equivnorm}
C_1 |Du|(\Omega) \leq  \Psi (D u)(\Omega) \leq C_2 |Du|(\Omega).
\end{equation}

%
%

\subsection{Properties of the functional $K_\eps$}

First we describe the properties of the polynomial that appears in \eqref{poly}. \par
For every $u\in L^1(\Omega)$, $\lambda, \mu \in \R$, and $Q^\prime=x_0+\eps Q$ the following properties hold true:
\begin{itemize}
	\item linearity: \[ \lambda P^{m-1}_{Q^\prime}[u](x)+ \eta P^{m-1}_{Q^\prime}[v](x)= P^{m-1}_{Q^\prime}[\lambda u+ \eta v](x);\]

	\item scaling: \[P^{m-1}_{\eps Q^\prime}[u](\eps x) = P^{m-1}_{Q^\prime}[u_\eps](x),\]
	where $u_\eps(x):=u(\eps x)$.
	\item convolution: for every open set $U$ compactly contained in $\O$ such that $\supp_\O u \subseteq U \subset\subset \O$ and for every $\sigma>0$, we consider  $u_\sigma=\rho_\sigma*u$ with $\rho$ a standard mollifier with compact support in the unit ball $B$ and $\rho_\sigma(x)=\sigma^{-n}\rho(x/\sigma)$. We choose $\sigma < \dist (\supp_\O u, \partial U)$ so that $\supp_\O u_\sigma \subset U$.  We have
	\begin{equation}\label{convP}
	P^{m-1}_{Q^\prime}[u_\sigma](x) = P^{m-1}_{Q^\prime}[\rho_\sigma*u](x) = \int_{B}\rho(y)P^{m-1}_{Q^\prime-\sigma y}[u](x) \,dy.
	\end{equation}
\end{itemize}

We recall that   
\[
BV^m(\O)=\{u\in W^{m-1,1}(\O),\:\, D^{m-1}u\in BV(\O, S^{m-1}(\R^n))  \}
\]
is the space of (real valued) functions of m-th order bounded variation, i.e. the set of all functions, whose distributional gradients up to order $m-1$ are represented through $1$-integrable tensor-valued functions and whose $m$-th distributional gradient is a tensor-valued Radon measure of finite total variation. Here $S^k(\R^n)$ denotes the set of all symmetric tensors of order $k$ with real components, which is naturally isomorphic to the set of all $k$-linear symmetric maps $(\R^n)^k \to \R$ (see \cite{DT}).

The space $BV^m(\Omega)$ becomes a Banach space with the norm 
\[
\|u\|_{BV^m(\O)} = \|u\|_{W^{m-1,1}(\O)} + |D^m u|(\O).
\]
Here the total variation of $D^{m-1}u$ is denoted by $|D^m u|(\O)$ and defined by
\[
|D^m u|(\O) = \sup \left( \sum_{\alpha_1, \dots, \alpha_m=1}^n \int_\O D_{\alpha_1, \dots, \alpha_{m-1}} u \cdot \partial_{\alpha_m}\varphi_{\alpha_1, \dots, \alpha_m} \, dx\right),
\]
where the supremum is taken over all $\varphi\in C^1_0(\O,\R^n)$ with $\|\varphi\|_\infty=1$. Obviously, $W^{m,1}(\Omega)$ is a subspace of $BV^m(\Omega)$. 

The definition of $BV^m$ generalizes that of the classical space of functions of bounded variation and many results about $BV$ can be obtained in $BV^m$ similarly (see \cite{JS}). We recall a higher-order variant of the famous Poincaré inequality, which will be useful throughout the sequel:
\begin{theo}[Poincarè inequality in $BV^m$ {\cite[Lemma 2.2]{FM}}]\label{PoincBVm}
	Let $\O\subset \R^n$ be an open and bounded subset with Lipschitz boundary, $m\in \mathbb{N}$, $1\leq p <\infty$. Then there exist a constant $C>0$, depending only on $\O$, $m$ and $n$ such that for all $u\in BV^m(\O)$
	\[
	\|u\|_{BV^m(\O)} \leq C |D^m u|(\O).
	\]
\end{theo}

In particular, the following version of Poincare's inequality holds.
Let $u\in BV^m(Q')$ with $Q'$ a translation of a cube of sidelenght $\eps$, then there exists a constant $C=C(n,m)$ such that
\be\label{BVpoinc}
\int_{Q'} |u-P^{m-1}_{Q'}[u]|\leq C \eps^m |D^m u|(Q'),
\ee
where $P^{m-1}_{Q'}[u]$ is defined in \eqref{poly}. This follows as in Lemma 2.1 of \cite{Marcus}.





Finally, given two functions $u, v\in L^1(\Omega)$, the following properties hold true:
\begin{itemize}
    \item an estimate from above:
    \be\label{upperK}
|K_\eps(u,m,\Omega)-K_\eps(v,m,\Omega)|\leq K_\eps(u-v,m,\Omega)
\ee
\item convexity: for any $\lambda\in [0,1]$, we have 
\be\label{convK}K_\eps(\lambda u+ (1-\lambda) v, m,\Omega) \leq \lambda K_\eps(u, m,\Omega)+(1-\lambda) K_\eps (v, m,\Omega) 
\ee
\end{itemize}




We summarize here some results of \cite{GS}:

\begin{theo}\label{sumarize}
It holds true that
\begin{itemize}
\item 
if $u\in W^{m,1}(\Omega)$, then 
\[
\lim_{\eps\to 0} K_\eps(u,m,\Omega)=\beta(m,n) \int_{\Omega}|\nabla^m u|\, dx\,,
\]
where the constant $\beta(m,n)$ is defined in \eqref{beta}. 
\item 
if $u\in W^{m-1,1}_{\text{loc}}(\Omega)$, then  
\be\label{charBVm}
	u\in BV^m(\Omega) \quad \Longleftrightarrow \quad \liminf_{\eps\to 0} K_\eps(u,m,\Omega)<+\infty.
	\ee
\end{itemize}
 \end{theo}
Moreover, there are  positive constants $C_1$ and $C_2$, independent of $u$, such that
	\[
 C_1|D^mu|(\Omega)\leq \liminf_{\eps\to 0^+} K_\eps(u,m,\Omega) \leq \limsup_{\eps\to 0^+} K_\eps(u,m,\Omega)\leq C_2 |D^m u|(\Omega).
	\]
\begin{rema}
In the case $m=1$ the constant $C_1=\frac{1}{4}$ was obtained in \cite{FMS2}. We will find the constant $C_1$ for $m>1$ in Proposition \ref{propHO}. 
\end{rema}

In the following statement, we prove that for $m=1$ the truncation of function $u\in L^{1}$ does not increase the oscillation in every cube $Q'\subset \mathbb R^n$ and this will serve its purpose in the upcoming stage. Notice that for $m>2$ the truncation could not preserve Sobolev space and the truncated function might not belong to $W^{m-1,1}$ anymore. We define the full truncation at level $k>0$ as 
 \be\label{truncfunc}
T_k(u)=\begin{cases} u & \text{if } |u|\leq k \\ k & \text{if } u>k \\ -k & \text{if } u<-k. \end{cases}
 \ee
 and prove the following.

\begin{prop}
\label{trunc_lemma}
    Let $m=1 $ and $u\in L^{1}_{\text{loc}} (\Omega)$, then
\begin{equation}\label{trunc_final}
    K_\eps(T_k(u), \Omega) \leq K_\eps(u,\Omega).
        \end{equation}
\end{prop}
\begin{proof}
Let $Q'\subset\Omega$ a cube centered in $x_0$. The polynomial $P^{0}_{Q'}[u]$ is described in \eqref{poly} and precisely
\begin{equation}\label{poly_1}   
P^0_{Q'}[u](x)= \fint_{Q'}u. 
\end{equation}
Firstly, we want to prove that 
\begin{equation}\label{trunc_eq}
\fint_{Q'} \left| T_k u- \fint_{Q'}T_k u \right| \leq \fint_{Q'} \left| u- \fint_{Q'}u \right|.
\end{equation}
Clearly, 
\begin{equation}\label{equali}
\fint_{Q'} \left|u- \fint_{Q'}u\right|= \frac{2}{|Q'|}\int_{\{u>\fint_{Q'}u\}} \left(u- \fint_{Q'}u\right)=\frac{2}{|Q'|}\int_{\{u<\fint_{Q'}u\}} \left( \fint_{Q'}u- u\right).
\end{equation}
For a fixed $k\in \mathbb R$ let us denote $u_1(x)=u(x)$, $u_2(x)=k$ and
define the truncation of $u$ from above at level $k$ by 
\[
g=\min (u_1,u_2)=\min(u,k).
\]
\\
Thus, let us consider 
\[
E_1=\{x\in Q^\prime\,:\, u_1(x)\leq u_2(x)\}=\{x\in Q^\prime\,:\, g(x)= u_1(x)\} \quad \text{and }\quad E_2=Q^\prime\setminus E_1.
\]
We have, by \eqref{equali}

\[
\begin{split}
\fint_{Q'} \left| g-\fint_{Q'}g \right| &= \frac{2}{|Q'|}\int_{ \{g<\fint_{Q'}g\} }\left(\fint_{Q'}g - g\right)= \frac{2}{|Q|} \sum_{i=1}^2 \int_{ \{x\in E_i: u_i<\fint_{Q'}g\} }\left(\fint_{Q'}g - u_i \right) \\ & \leq \frac{2}{|Q'|} \sum_{i=1}^2 \int_{ \{x\in Q': u_i<\fint_{Q'}u_i\} }\left(\fint_{Q'}u_i -u_i \right) = \fint_{Q'} \left|u_1-\fint_{Q'}u_1\right|.
\end{split}
\]

Similarly, define the truncation from below at level $k$ by 
\[h=\max (u_1,u_2)=\max (u,k)\]
 and we have
 $$\fint_{Q'} \left| h-\fint_{Q'}h \right| \leq \fint_{Q'} \left|u_1-\fint_{Q'}u_1\right|.$$
 Since
 \[
 T_k u = \max (\min (u,k),-k),
 \] \eqref{trunc_eq} holds true and summing up over all cubes in the family $\mathcal G_\eps$ we conclude.
\end{proof}

\section{$\Gamma-$convergence for $H_\eps$ and $K_\eps$}

We recall (see for example \cite{DM}) that a family of functionals $\F_\eps$ defined on $L^1(\Omega)$, $\Gamma$-converges in $L^1(\Omega)$ when $\eps$ goes to $0$ to a functional $\F$ defined on $L^1(\Omega)$ if and only if, the following two conditions are satisfied:
\begin{itemize}
	\item[i)] (\textit{$\Gamma-$liminf inequality}) $\forall u \in L^1(\Omega)$ and for every family $\{u_\eps\}$ such that $u_\eps \to u$ in $L^1(\Omega)$ as $\eps$ goes to $0$, one has
	\[ \liminf_{\eps\to 0} \F_\eps(u_\eps) \geq \F(u);
	\]
	\item[ii)] (\textit{$\Gamma-$limsup inequality}) $\forall u \in L^1(\Omega)$, there exists a family $\{\tilde{u}_\eps\}$ such that $\tilde{u}_\eps$ converges to $u$ in $L^1(\Omega)$ as $\eps$ goes to $0$ and
	\[ \limsup_{\eps\to 0} \F_\eps(\tilde{u}_\eps) \leq \F(u).
	\]
\end{itemize}

We give now the proofs of Theorem \ref{gammaH} and Theorem \ref{gammaK}. 

\begin{proof}[Proof of Theorem \ref{gammaH}]
	We assume without loss of generality that $ \Psi(D u)(\Omega)<+\infty$ and then by \eqref{equivnorm} $f\in BV(\Omega)$. To prove the $\Gamma-$limsup inequality by \eqref{star} we consider a sequence $(v_h)\in C^\infty(\Omega)$ such that 
 \[
 v_h \to u \quad \text{in } L^1(\Omega) \quad \text{when } h \text{ goes to } 0 
 \]
 and
 \[
 \int_\Omega|\nabla v_h| \to |Du|(\Omega).
 \]
 By Reshetnyak’s continuity theorem, it holds that
\be\label{lim""}
 \int_\Omega \psi(\nabla v_h) \to \Psi(Du)(\Omega).
 \ee
 By \eqref{w11lim1}, for each $h$ let $\eps_h$ be such that 
 \be\label{lim}
 \left|H_\eps(v_h,\Omega)- \int_\Omega\psi(\nabla v_h)\right|<h \qquad \forall \eps < \eps_h.
 \ee
 Without loss of generality we may assume that $\eps_h$ is an infinitesimal decreasing sequence with respect to $h$. We set
 \[
 u_\eps= v_h \qquad \text{if } \eps_{h+1}< \eps \leq \eps_h.
 \]
 Combining \eqref{lim""} and \eqref{lim}, it holds 
 \[
 u_\eps \to u \quad \text{in } L^1(\Omega)
 \]
 and 
 \[
 \lim_{\eps\to 0} H_\eps(u_\eps,\Omega) = \Psi(Du)
 \]
 proving the $\Gamma-$limsup inequality.

 To prove the $\Gamma-$liminf inequality we use the convolution strategy by considering $u_\sigma=u*\rho_\sigma$ and by fixing an open set $A$, $A\subset\subset \Omega$ as in Lemma \ref{lemmaconvol}. 
 If $u\in BV$ then by $a)$ of Theorem \ref{w11lim1}, we have
	\[
	\lim_{\eps\to 0} H_\eps(u*\rho_\sigma, A) = \int_{A} \psi\big(\nabla( u*\rho_\sigma)\big)\, dx.
	\]
Moreover, by using \eqref{upperdiff} and the Poincarè inequality \eqref{Poinc}, we get
	\be\label{est2}
 \begin{split}
	|H_\eps(u_\eps*\rho_\sigma,A)- &H_\eps(u*\rho_\sigma,A)|\\&\leq C\int_A |\nabla\big( (u_\eps-u)*\rho_\sigma\big)|\\&\leq C \|\nabla \rho_\sigma\|_{L^1(A)} \|u_\eps-u\|_{L^1(A)},
\end{split}		
 \ee
	where $C$ depends only on $A$ and the dimension $n$.
	Then applying \eqref{convol} and \eqref{est2}, we have
	\be\label{due}
	\liminf_{\eps\to 0}H_\eps(u_\eps,A)\geq \lim_{\eps\to 0}H_\eps(u_\eps*\rho_\sigma,A)= \int_{A} \psi(\nabla u*\rho_\sigma)\, dx.
 \ee
Letting $\sigma\to 0$, by Definition \ref{defPsi} we have
 \be\label{uno}
\lim_{\sigma\to 0} \int_{A} \psi\big(\nabla (u*\rho_\sigma)\big)\, dx \geq \Psi(Du)(A).
 \ee
Thus, for every $u_\eps\to u$, combining \eqref{uno} and \eqref{due}, letting $A\uparrow \Omega$, we get
	\[
	\Psi(D u)(\Omega)\leq \liminf_{\eps\to 0} H_\eps(u_\eps,\Omega).
	\]
 proving the $\Gamma-$liminf inequality.
\end{proof}


We prove now a lower bound for the functional $K_\eps(u,m,\Omega)$ for $m\geq 2$. The case $m=1$ has been proved in Proposition 3.4 in \cite{FMS2}.

\begin{prop}\label{propHO}
	Let $m\geq 2$ and $u\in W^{m-1,1}_{\text{loc}}(\Omega)$. Then
	\[
	\liminf_{\eps\to 0^+} K_\eps(u,m,\Omega) \geq \beta(m,n)|D^mu|(\Omega)
	\]
\end{prop}
\begin{proof}
Without loss of generality we can assume that
$$\liminf_{\eps\to 0^+} K_\eps(u,m,\Omega)<+ \infty.$$ By Theorem \ref{sumarize} we have that $u\in BV^m$. We fix an open set $A\subset\subset \Omega$, $\sigma>0$ and we set $u_\sigma=\rho_\sigma*u$, as in the proof of Lemma \ref{lemmaconvol}. Then, by using the definition of $u_\sigma$, Jensen's inequality and Fubini's Theorem, we have that
\be\label{usigmautilde}
\begin{split}
	\eps^{n-m} \sum_{Q'\in \mathcal G_\eps}\fint_{Q^\prime} \Big| &u_\sigma(x) -  P^{m-1}_{Q^\prime} [u_\sigma](x) \Big| dx \\& =\eps^{n-m} \sum_{Q'\in \mathcal G_\eps}  \fint_{Q^\prime} \left| \int_{B}\rho(y)\left( \tilde{u}(x) -P^{m-1}_{Q^\prime} [\tilde{u}](x)\right) dy\right| dx    \\ &
	\leq \eps^{n-m} \int_B \rho(y)  \left( \sum_{Q'\in \mathcal G_\eps} \fint_{Q^\prime}\left|\tilde{u}(x)- P^{m-1}_{Q^\prime} [\tilde{u}](x)\right| dx\right)dy,
	\end{split}
 \ee
where $\Tilde{u}(x)=u(x-\sigma y)$. 
Then, a change of variable shows that
\[
P_{Q^\prime}^{m-1}[\tilde{u}](x)=P_{Q^\prime-\sigma y}^{m-1}[u](x-\sigma y)
\]
and since $\int_B \rho(y)dy =1$,  we deduce 
\be\label{Kconvol}
\eps^{n-m} \sum_{Q'\in \mathcal G_\eps}\fint_{Q^\prime} \Big| u_\sigma(x) -  P^{m-1}_{Q^\prime} [u_\sigma](x) \Big| dx \leq K_\eps(u,m,\Omega).
\ee
Since $u_\sigma \in W^{m,1}(A)$ from Theorem \ref{theoGS} we have
	\[
	\beta(m,n)\int_{A}|\nabla^mu_\sigma|= \liminf_{\eps\to 0}K_\eps(u_\sigma,m,A)\leq \liminf_{\eps\to 0} K_\eps(u,m,\Omega).
	\]
	We conclude letting $\sigma \to 0$ and then $A\uparrow \Omega$.
\end{proof}

\begin{proof}[Proof of Theorem \ref{gammaK}]
	To prove the $\Gamma-$limsup inequality we proceed in the same way as in Theorem \ref{gammaH} by replacing \eqref{lim""} with the convergence of total variation in higher dimension and \eqref{lim} with the analogous result holding by Theorem \ref{theoGS}. 

 To prove the $\Gamma-$liminf inequality, from \eqref{Kconvol} in Proposition \ref{propHO} we have
	\[
	K_\eps(u*\rho_\sigma,m,A)\leq K_\eps(u,m,\Omega)
	\]
	and by Theorem \ref{theoGS}
	\[
	\lim_{\eps\to 0} K_\eps(u*\rho_\sigma,m,A) = \beta(m,n)\int_{A} |\nabla^m(u*\rho_\sigma)|\, dx.
	\]
	Moreover, by using \eqref{upperK} and the Poincarè inequality in $BV^m$ (equation \eqref{BVpoinc}), we have
	\[
	|K_\eps(u_\eps*\rho_\sigma,m,A)-K_\eps(u*\rho_\sigma,m,A)|\leq C \|\nabla^m \rho_\sigma\|_{L^1(A)} \|u_\eps-u\|_{L^1(A)} .
	\]
	Then combining these last three inequalities we have
	\[
	\liminf_{\eps\to 0}K_\eps(u_\eps,m,\Omega)\geq \lim_{\eps\to 0}K_\eps(u_\eps*\rho_\sigma,m,A)= \beta(m,n)\int_{A} |\nabla^m(u*\rho_\sigma)|\, dx.
	\]
	Then, letting $\sigma\to 0$ and then $A\uparrow \Omega$, we prove that for every $u_\eps\to u$,
	\[
	\liminf_{\eps\to 0} K_\eps(u_\eps,m,\Omega) \geq \beta(m,n) |D^mu|(\Omega).
	\]
\end{proof}

\section{Application to image processing}
Having an image of poor quality, a challenging problem is to find a better one not so far from the original. A popular strategy in image processing to improve an initial image $f$ is to use a variational formulation by considering the problem
\be\label{vp}
\inf \left\{ F(u)+ \Lambda \int_\Omega |f-u|^2 \,:\, u\in \mathcal{A} \right\}.
\ee
Here $\mathcal{A}$ is a suitable functional space, $\Lambda>0$ is the fidelity parameter which fixed the grade of similarity with the original $f$ and the functional $F$, called ﬁlter, has a regularization role. Minimizers of \eqref{vp} are the better images.    

Many kind of filters have been proposed starting from the most famous one by Rudin, Osher and Fatemi (see \cite{ROF}) where 
\[
F(u)=|D u|(\Omega)
\]
and the minimization problem is  
\be\label{rof}
\min \left\{ |D u|(\Omega)+ \int_\Omega |f-u|^2 \,:\, u\in L^2(\Omega) \right\}. \tag{ROF}
\ee
The advantages of the minimization model \eqref{rof} are mainly the fact that the BV regularization term allows for discontinuities but disfavors large oscillations and moreover the strict convexity due to the $L^2$ approximation term uniquely determines the minimizers $u$ in terms of the datum $f$ and the chosen fidelity $\Lambda>0$.

On the other hand, the ROF model is not contrast invariant, i.e. if $u$ is the solution of \eqref{rof} for the initial $f$, then $cu$ may not be the solution for $cf$. This relies on the $L^2$ norm in the approximation term.  A model that closes this lack of contrast invariance was proposed by Chan and Esedoglu (see \cite{CE}) by considering the model
\[
\inf \left\{ |D u|(\Omega)+  \Lambda\int_\Omega |f-u| \,:\, u\in BV(\Omega) \right\} \tag{CE}
\]
where the $L^1$ norm appears in the fidelity term. Nevertheless, this model is only convex thus the uniqueness of minimizers is not guaranteed. 

Other variants have been proposed, by considering families of filters varying with respect to a small parameter $\eps$. In \cite{AK} the authors proposed as filter a non local functional modeled on one considered by Bourgain, Brezis, Mironescu in \cite{BBM}
\[
AK_\eps(u)= \int_\Omega\int_\Omega \frac{|u(x)-u(y)|}{|x-y|}\rho_\eps(|x-y|)\, dxdy.
\]
In \cite{GO} the non local filter includes a given weight function:
\[
GO(u)= \int_\Omega\left(\int_\Omega \frac{|u(x)-u(y)|^2}{|x-y|^2}w(x,y)\right)^{1/2}\, dxdy,
\]
and by choosing $w(x,y)=\rho_\eps(|x-y|)$ one has the corresponding family $GO_\eps$ of filters.

Considering families of filters the aim is twofold: first, one investigates the existence of a minimizer for the approximated functional when $\eps$ is fixed and then one can study the behavior of these minimizers 
as $\eps \to 0$. It is proved that the minimizers of $AK_\eps$ and $GO_\eps$ converge to the unique solution of the problem 
\be\label{rof2}
\inf \left\{ k |D u|(\Omega) + \Lambda\int_\Omega |f-u|^2 \,:\, u\in L^2(\Omega) \cap BV(\Omega)\ \right\}   \tag{$ROF_k$}
\ee
with the constant $k=\int_{\mathbb{S}^{d-1}}|\sigma\cdot e|\,d\sigma$ and $k=\left(\int_{\mathbb{S}^{d-1}}|\sigma\cdot e|^2\,d\sigma\right)^{1/2}$ respectively.
We remark that the functional in \eqref{rof2} is strictly convex and by standard functional analysis the solution is unique. 

Here, in the spirit of functionals converging to the BV-norm considered in \cite{BourNguyen, BreNguyen, BreNguyen2, AK}, the following functional is studied
\be\label{gseps}
F_\eps(u,\Omega) = K_\eps(u,\Omega)+ \Lambda \int_\Omega |f-u|^q, \quad q\geq 1 
\ee
where for brevity we denote $K_\eps(u,\Omega)=K_\eps(u,1,\Omega)$. We deal with the corresponding minimization problem
\[
\inf \left\{ F_\eps(u,\Omega) \,:\, u\in L^q(\Omega) \right\}.
\]
For $q>1$, we claim that minimizers as $\eps \to 0$ converge to the solution of the problem
\[
\inf \left\{ F(u,\Omega) \,:\, u\in L^q(\Omega) \cap BV(\Omega) \right\},  \tag{$ROF^q$}
\]
where 
\[
F(u,\Omega)=\frac{1}{4} |D u|(\Omega) + \Lambda\int_\Omega |f-u|^q.
\]

We are now in position to prove Theorem \ref{mainappl}.
\begin{proof}[Proof of Theorem \ref{mainappl}]
	For every $\eps >0$, the functional $F_\eps$ defined on $L^q(\Omega)$ in \eqref{gseps} is convex, since $K_\eps(u,\Omega)$ is convex as showed in \eqref{convK}. Moreover, $F_\eps$ is lower semicontinuous for the strong $L^q$ topology. Indeed, given $u_n\to u$ strongly in $L^q(\Omega)$, it converges also in $L^1(\Omega)$ and by Fatou's Lemma, we have
	\[
	\liminf_{n\to +\infty} K_\eps(u_n,\Omega) \leq K_\eps(u,\Omega).
	\]
	Therefore $F_\eps$ is also lower semicontinuous for the weak topology in $L^q(\Omega)$. When $q>1$ the space $L^q(\Omega)$ is reflexive and then there exists a minimizer $u_\eps$ of $F_\eps$ in $L^q(\Omega)$. The uniqueness follows by the strict convexity. 

 In order to prove that $u_\eps\to u_0$, since $q>1$, we assume that there exists a subsequence $u_{\eps_k}$ such that $u_{\eps_k} \rightharpoonup v$ weakly in $L^q(\Omega)$. We will prove that $v=u_0$ and we divide the proof in two steps.

{\it Step 1.} We first observe that, since the limsup inequality holds true in Theorem \ref{gammaK} for $K_\eps$, there exists $v_\eps$ contained in $L^1(\Omega)$ such that $v_\eps \to u_0$ in $L^1(\Omega)$ and
 \[
 \limsup_{\eps\to 0} K_\eps (v_\eps,\Omega) \leq \frac{1}{4} |Du_0|(\Omega).
 \]
 Let us consider, for $\tau>0$ the truncation function $T_\tau(\cdot)$ defined in \eqref{truncfunc}. Since $u_\eps$ is a minimizer of $F_\eps$ we have
\be\label{uepsminim}
F_\eps(u_\eps,\Omega) \leq F_\eps(T_\tau v_\eps,\Omega)\leq K_\eps(T_\tau v_\eps,\Omega)+ \Lambda \int_\Omega |f-T_\tau v_\eps|^q. 
\ee
Moreover, by Proposition \ref{trunc_lemma}, as the truncation may not increase the oscillation in every cube $Q^\prime$, for every $w$
\be\label{trunest}
\fint_{Q^\prime} \left|T_\tau w - \fint_{Q^\prime} T_\tau w\right|\, dx \leq \fint_{Q^\prime} \left| w - \fint_{Q^\prime} w\right|\, dx.
\ee
Combining \eqref{uepsminim} and \eqref{trunest} we have
\[
F_\eps(u_\eps,\Omega) \leq K_\eps(v_\eps,\Omega)+ \Lambda \int_\Omega |f-T_\tau v_\eps|^q.  
\]
Letting $\eps\to 0$ and using the limsup inequality, we have
\[
 \limsup_{\eps\to 0} F_\eps(u_\eps,\Omega) \leq \frac{1}{4} |Du_0|(\Omega)+ \Lambda \int_\Omega |f-T_\tau u_0|^q.
\]
Moreover, letting $\tau\to +\infty$,
\be\label{limsupeps}
 \limsup_{\eps\to 0} F_\eps(u_\eps,\Omega) \leq \frac{1}{4} |Du_0|(\Omega)+ \Lambda \int_\Omega |f- u_0|^q = F(u_0,\Omega).
\ee

{\it Step 2.}  We claim that  for every $w\in L^1(\Omega)$ and for every sequence $w_\eps \in L^1(\Omega)$ such that $w_\eps \rightharpoonup w$ weakly in $L^1(\Omega)$ we have
\be\label{g1weak}
 \liminf_{\eps\to 0} K_\eps(w_\eps,\Omega) \geq  \frac{1}{4} |D w|(\Omega).
\ee
We fix fix an open set $A\subset\subset \Omega$, $\sigma>0$ and we set $w_{\sigma,\eps}(x)=(\varrho_\sigma*w_\eps)(x)$, as in the proof of Lemma \ref{lemmaconvol}.
Since $w_\eps\rightharpoonup w$ weakly in $L^1(\Omega)$, for each fixed $\sigma>0$,
\[
w_{\sigma,\eps}=\varrho_\sigma*w_\eps \to \varrho_\sigma*w\ = w_\sigma\,\, \text{ strongly in } L^1(A) \text{ as } \eps \to 0. 
\]
Moreover by \eqref{Kconvol} and the $\Gamma$-liminf inequality for $K_\eps$, we have
\[
 \liminf_{\eps\to 0} K_\eps(w_\eps,\Omega) \geq  \liminf_{\eps\to 0} K_\eps (w_{\sigma,\eps},A) \geq \frac{1}{4} |Dw_\sigma|(A). 
\]
Letting $\sigma\to 0$ and the $A\uparrow \Omega$,  we conclude proving \eqref{g1weak}.

Now, applying this claim to the sequence $u_{\eps_k}\rightharpoonup v$, we have
\[
 \liminf_{\eps_k\to 0} K_{\eps_k}(u_{\eps_k},\Omega) \geq \frac{1}{4} |Dv|(\Omega)
\]
and therefore
\be\label{liminfeps}
 \liminf_{\eps_k\to 0} F_{\eps_k}(u_{\eps_k},\Omega) \geq \frac{1}{4} |Dv|(\Omega)+ \Lambda \int_\Omega |f- v|^q = F(v,\Omega).
\ee
Combining \eqref{liminfeps} and \eqref{limsupeps}, by the uniqueness of minimizers, we have
\[
u_0=v.
\]
It remains to prove that $F_\eps(u_\eps,\Omega)\to F(u_0,\Omega)$. We write
\[
\Lambda \int_\Omega |f- u_\eps|^q = F_\eps(u_\eps,\Omega) - K_\eps(u_\eps,\Omega) 
\]
and we use \eqref{liminfeps} and $\Gamma$-limsup inequality for  $K_\eps$ to have
\be\label{uepstrong}
 \limsup_{\eps_k\to 0} \Lambda \int_\Omega |f- u_{\eps_k}|^q \leq F(u_0,\Omega) - \frac{1}{4} |Du_0|(\Omega)=\Lambda \int_\Omega |f- u_0|^q.
\ee
We know that  $u_{\eps_k}$ such that $u_{\eps_k} \rightharpoonup u_0$ weakly in $L^q(\Omega)$ so by \eqref{uepstrong} $u_{\eps_k} \to u_0$ strongly in $L^q(\Omega)$. This implies $u_{\eps} \to u_0$ strongly in $L^q(\Omega)$ and then
\[
\liminf_{\eps\to 0} F_{\eps}(u_{\eps},\Omega)\geq \frac{1}{4} |Du_0|(\Omega)+ \Lambda \int_\Omega |f- u_0|^q =F(v,\Omega)
\]
which combined with \eqref{limsupeps} conclude the proof.  
 \end{proof}

In Theorem \ref{mainappl}, one cannot replace the condition $q>1$ by $q=1$. In this case we prove a slight generalization result by consider almost minimizers.

\begin{theo}\label{quasimin}
Let $q=1$ and $\Omega\subseteq \R^n$ be a smooth bounded open set. Let $f\in L^1(\Omega)$ and let $\{\delta_\eps\}$, $\{\tau_\eps\}$ two positive sequences converging to 0 as $\eps\to 0$. Let $\{u_{\eps}\}\in L^1(\Omega)$ equi-bounded in $L^1(\Omega)$ such that
	\[
	F_{\delta_\eps}(u_{\eps},\Omega)\leq \inf_{u\in L^1(\Omega)} F_{\delta_\eps}(u,\Omega) + \tau_\eps.
	\]
Then there exists a subsequence $\{u_{\eps_k}\}$ of $\{u_{\eps}\}$ such that $u_{\eps_k}$ converges to $u_0$ in $L^1(\Omega)$ where $u_0$ is a minimizer of the functional $F$ defined on $L^1(\Omega)\cap BV(\Omega)$. 
\end{theo}
\begin{proof}
We immediately observe that $\liminf_{\eps \to 0} K_\eps(u_\eps)<\infty$. Repeating the argument as in Theorem 2 of \cite{ARB} we can consider an approximating sequence $u_\eps^{\eps/2}$ as defined in (19) of \cite{ARB}. By Lemma 9, 10 and 11 in \cite{ARB}, $u_\eps^{\eps/2}$ has uniformly bounded total variation and $u_\eps-u_\eps^{\eps/2}\to 0$ in $L^1(\Omega)$. Then, by BV compactness theorem \cite[Thm. 3.23]{AFP} there exists a subsequence $\eps_k$ such that $u_{\eps_k}^{\eps_k/2}$ converges to some $u_0$ and then $u_{\eps_k}\to u_0$  in $L^1(\Omega)$. By Fatou's Lemma and the $\Gamma$-liminf property of $K_\eps$ it follows
	\[
	F(u_0,\Omega)\leq \liminf_{k\to \infty} F_{\delta_{\eps_k}}(u_{\eps_k},\Omega).
	\]
	We prove now that $u_0$ is a minimizer of $F$ in $L^1(\Omega)\cap BV(\Omega)$. Let $v\in L^1(\Omega)\cap BV(\Omega)$ be a minimizer of $F$. Applying Theorem \ref{gammaK}  there exists $v_\eps\in L^1(\Omega)$ such that $v_\eps\to v$ in $L^1(\Omega)$ and by the $\Gamma$-limsup inequality,
	\[
	\limsup_{\eps\to 0} K_{\delta_\eps}(v_\eps,\Omega) \leq \frac{1}{4} |Dv|(\Omega).
	\]
	For $A>0$, let $T_A v$ the truncation of $v$ at level $A$ defined in \eqref{truncfunc}. By Proposition \ref{trunc_lemma} we have
	\[
	K_{\delta_\eps}(T_A v_\eps,\Omega) \leq K_{\delta_\eps}(v_\eps,\Omega)
	\]
	and by definition of $u_\eps$, we get
	\begin{multline}\label{est1}
	F_{\delta_\eps}(u_\eps,\Omega)\leq K_{\delta_\eps}(T_A v_\eps,\Omega) + \tau_\eps + \Lambda \int_\Omega |T_A v_\eps -f| \\ \leq K_{\delta_\eps}(v_\eps,\Omega) + \tau_\eps + \Lambda \int_\Omega |T_A v_\eps -f|.
	\end{multline}
	Taking first the $\liminf$ as $\eps\to 0$ and then let $A\to +\infty$, it follows
	\[
	F(u_0,\Omega)\leq  \frac{1}{4} |Dv|(\Omega) + \Lambda \int_\Omega |v -f|
	\]
	which implies that $u_0$ is a minimizer of $F$.\\
\end{proof}

\begin{rema}
Let us note that in the previous Theorem it is possible to choose a sequence of almost minimizers which is equibounded in $L^1(\O)$, at least for smooth datum $f$. Indeed, if $f\in L^\infty(\O)$ and $\{u_\eps\}$ is a sequence of almost minimizers, then the sequence $\{T_{\|f\|_\infty}(u_\eps)\}$ of truncation at level $\|f\|_\infty$ is equibounded in $L^1(\O)$ and a still almost minimizing. 
\end{rema}

\bigskip

\noindent{\bf Acknowledgments.} The research of S.GLB. has been funded by PRIN Project 2022ZXZTN2. 
The research of R.S. has been partially funded by PRIN Project 2022XZSAFN and PRIN-PNRR P2022XSF5H. The authors  are members of "GNAMPA" of Istituto Nazionale di Alta Matematica (INdAM),  and they were  partially supported by the project 2024 "Problemi variazionali associati alla variazione totale".

\bigskip
\small\noindent
Serena Guarino Lo Bianco:\\
Dipartimento di Scienze F.I.M., Universit\`a degli Studi di Modena e Reggio Emilia\\
Via Campi 213/A, 41125 Modena - ITALY\\
{\tt serena.guarinolobianco@unimore.it}\\

\bigskip
\small\noindent
Roberta Schiattarella:\\
Dipartimento di Matematica e Applicazioni "R. Caccioppoli", Universit\`a di Napoli ``Federico II"\\
via Cintia, 80126 Napoli, ITALY\\
{\tt roberta.schiattarella@unina.it}\\

\end{document}